\documentclass[ECP]{ejpecp} 

\SHORTTITLE{Note on A. Barbour's paper on Stein's method for diffusion approximations} 

\TITLE{Note on A. Barbour's paper on Stein's method for diffusion approximations} 

\AUTHORS{%
  Miko{\l}aj~J.~Kasprzak\footnote{University of Oxford, United Kingdom.
    \EMAIL{kasprzak@stats.ox.ac.uk}}
    \and
 Andrew~B.~Duncan\footnote{University of Sussex, United Kingdom.
    \EMAIL{andrew.duncan@sussex.ac.uk}}   
  \and 
  Sebastian~J.~Vollmer\footnote{University of Warwick,
    United Kingdom. \EMAIL{svollmer@turing.ac.uk}}}


\AMSSUBJ{Primary: 60B10, 60F17, Secondary: 60J60, 60J65, 60E05} 

\SUBMITTED{February 22, 2017} 
\ACCEPTED{April 4, 2017} 

\VOLUME{22}
\YEAR{2017}
\PAPERNUM{23}
\DOI{54}

\ABSTRACT{In \cite{diffusion} foundations for diffusion approximation via Stein's method are laid. This paper has been cited more than 130 times and is a cornerstone in the area of Stein's method (see, for example, its use in \cite{barbour} or \cite{holmes}). A semigroup argument is used in \cite{diffusion} to solve a Stein equation for Gaussian diffusion approximation. We prove that, contrary to the claim in \cite{diffusion}, the semigroup considered therein is not strongly continuous on the Banach space of continuous, real-valued functions on $D[0,1]$ growing slower than a cubic, equipped with an appropriate norm. We also provide a proof of the exact formulation of the solution to the Stein equation of interest, which does not require the aforementioned strong continuity. This shows that the main results of \cite{diffusion} hold true.}

\begin{document}



\section{Introduction}
In \cite{diffusion} a claim is made that the semigroup defined by (2.4) thereof is strongly continuous on space $L$ defined on page 299 thereof. We prove that this is not the case. Nevertheless, we show that the only assertion of the paper following from the aforementioned assumption of strong continuity, namely the claim that (2.20) solves the Stein equation (2.1), remains true. This may be proved by adapting the proof of \cite[Proposition 9, p. 9]{ethier} and noting that in the case of interest in \cite{diffusion}, the point-wise continuity of the semigroup is sufficient. It then follows that all the other results of \cite{diffusion} hold true.

In Section \ref{defs} we recall the relevant definitions and notation from \cite{diffusion}. In Section \ref{count} we give a counterexample to the strong continuity of the semigroup. In Section \ref{sol} we provide a proof of the fact that the function (2.20) of \cite{diffusion} does actually solve the Stein equation. We do this by following the steps of the proof of \cite[Proposition 9, p. 9]{ethier} and proving each of the assertions therein  for the semigroup of interest by hand. 

\section{Definitions and notation}\label{defs}
By $D=D[0,1]$ we will mean the Skorohod space of all the c\`adl\`ag functions\\ $w:[0,1]\to\mathbb{R}$. In the sequel $\|\cdot\|$ will always denote the supremum norm. By $D^kf$ we mean the $k$-th Fr\'echet derivative of $f$ and the $k$-linear norm $B$ is defined to be $\|B\|=\sup_{\lbrace h:\|h\|=1\rbrace} |B[h,...,h]|$. We will also often write $D^2f(w)[h^{(2)}]$ instead of $D^2f(w)[h,h]$. Let:
$$L=\left\lbrace f:D\to\mathbb{R}:f\text{ is continuous and }\sup_{w\in D}\frac{|f(w)|}{1+\|w\|^3}<\infty\right\rbrace$$
and for any $f\in L$ let $\|f\|_L=\sup_{w\in D}\frac{|f(w)|}{1+\|w\|^3}$.

We define:
\begin{align*}
\|f\|_M=\sup_{w\in D}\frac{|f(w)|}{1+\|w\|^3}+\sup_{w\in D}\frac{\|Df(w)\|}{1+\|w\|^2}+\sup_{w\in D}\frac{\|D^2f(w)\|}{1+\|w\|}+\sup_{w,h\in D}\frac{\|D^2f(w+h)-D^2f(w)\|}{h}
\end{align*}
for any $f\in L$ for which the expressions exist and 
$$M=\left\lbrace f\in L:f\text{ is twice Fr\'echet differentiable and }\|f\|_M<\infty\right\rbrace.$$

The Stein operator for approximation by $Z$, the Brownian Motion on $[0,1]$, is defined, as in (2.9) and (2.11) of \cite{diffusion}, by:
$$\mathcal{A}f(w)=-Df(w)[w]+\mathbb{E}D^2f(w)\left[Z^{(2)}\right]=-Df(w)[w]+\sum_{k\geq 0} D^2f(w)\left[S_k^{(2)}\right],$$
for any $f:D[0,1]\to\mathbb{R}$, for which it exists. By $(S_k)_{k\geq 0}$ we denote the Schauder functions defined, as on page 299 of \cite{diffusion} by:
$$S_0(t)=t;\qquad S_k(t)=\int_0^t H_k(u) du,\quad k\geq 1,$$
where, for $2^n\leq k<2^{n+1}$:
$$H_k(u)=2^{n/2}\left(\mathbb{1}\left[\frac{k}{2^n}-1\leq u\leq \frac{k+\frac{1}{2}}{2^n}-1\right]-\mathbb{1}\left[\frac{k+\frac{1}{2}}{2^n}-1<u\leq \frac{k+1}{2^n}-1\right]\right).$$

We also define a semigroup acting on $L$:
\begin{equation}\label{semigroup}
(T_uf)(w)=\mathbb{E}\left[f\left(we^{-u}+\sigma(u)Z\right)\right],
\end{equation}
where $\sigma^2(u)=1-e^{-2u}$.

For any $g\in M$ with $\mathbb{E}g(Z)=0$, the Stein equation is given by:
$$\mathcal{A}f=g.$$
The idea of Stein's method applied in \cite{diffusion} is to find a bound on $\mathbb{E}\mathcal{A}f(X)$, where $f$ is a solution to this equation, in order bound $|\mathbb{E}g(X)-\mathbb{E}g(Z)|$, for some stochastic process $X$ on $[0,1]$.
\section{Counterexample to strong continuity}\label{count}
It is well known that the Ornstein-Uhlenbeck semigroup is not strongly continuous on the space $C_b(\mathbb{R})$, see \cite{daprato1995ornstein}.   More generally, given a separable Hilbert space $H$, in \cite{manca2008kolmogorov} it is noted that this semigroup is also not strongly continuous on the space $C_{b,k}$ of all continuous functions $\psi:H\rightarrow \mathbb{R}$ such that $x \rightarrow \psi(x)/(1+|x|^k)$ is uniformly continuous and $\sup_{x\in H}\frac{\psi(x)}{1 + |x|^k} < \infty$.   Following these two results, in this section we shall show that the semigroup $T_u$ defined by (\ref{semigroup}) is not strongly continuous on the Banach space $L$ by constructing an explicit counterexample.
\begin{lemma}
The semigroup $T_u$ is not strongly continuous on $\left(L,\|\cdot\|_L\right)$.
\end{lemma}
\begin{proof}
Consider $f\in L$ defined by:
$$f(w)=(1+\|w\|^3)\sin\left(\|w\|\right).$$
Note that:
\begin{align}
\|T_uf-f\|_L=&\sup_{w\in D}\frac{\left|\mathbb{E}(1+\|we^{-u}+\sigma(u)Z\|^3)\sin(\|we^{-u}+\sigma(u)Z\|)-(1+\|w\|^3)\sin(\|w\|)\right|}{1+\|w\|^3}\nonumber\\
=&\sup_{w\in D}\left|\mathbb{E}\sin(\|we^{-u}+\sigma(u)Z\|)-\sin(\|w\|)\right.\nonumber\\
&\left.+\frac{\mathbb{E}\left[\left(\|we^{-u}+\sigma(u)Z\|^3-\|w\|^3\right)\sin(\|we^{-u}+\sigma(u)Z\|)\right]}{1+\|w\|^3}\right|\nonumber\\
\geq&\sup_{w\in D}\left|\mathbb{E}\sin(\|we^{-u}+\sigma(u)Z\|)-\sin(\|w\|)\right|\nonumber\\
&-\sup_{w\in D}\left|\frac{\mathbb{E}\left[\left(\|we^{-u}+\sigma(u)Z\|^3-\|w\|^3\right)\sin(\|we^{-u}+\sigma(u)Z\|)\right]}{1+\|w\|^3}\right|\nonumber\\
\geq& \sup_{w\in D}\left|\sin(e^{-u}\|w\|)-\sin(\|w\|)\right|-\sup_{w\in D}\left|\mathbb{E}\sin(\|we^{-u}+\sigma(u)Z\|)-\sin(e^{-u}\|w\|)\right|\nonumber\\
&-\sup_{w\in D}\left|\frac{\mathbb{E}\left[\left(\|we^{-u}+\sigma(u)Z\|^3-\|w\|^3\right)\sin(\|we^{-u}+\sigma(u)Z\|)\right]}{1+\|w\|^3}\right|\label{1}.
\end{align}
Now:
\begin{align}
&\sup_{w\in D}\left|\frac{\mathbb{E}\left[\left(\|we^{-u}+\sigma(u)Z\|^3-\|w\|^3\right)\sin(\|we^{-u}+\sigma(u)Z\|)\right]}{1+\|w\|^3}\right|\nonumber\\
\leq&\sup_{w\in D}\frac{\mathbb{E}\left|\left(\|we^{-u}+\sigma(u)Z\|-\|w\|\right)\left(\|we^{-u}+\sigma(u)Z\|^2+\|we^{-u}+\sigma(u)Z\|\|w\|+\|w\|^2\right)\right|}{1+\|w\|^3}\nonumber\\
\leq& \sup_{w\in D}\frac{\mathbb{E}\left[\left(\|w\|(1-e^{-u})+\sigma(u)\|Z\|\right)\left(\|w\|^2(2e^{-2u}+e^{-u}+1)+\sigma(u)\|Z\|\|w\|+2\sigma^2(u)\|Z\|^2\right)\right]}{1+\|w\|^3}\nonumber\\
=&\sup_{w\in D}\frac{1}{1+\|w\|^3}\cdot\left\lbrace\|w\|^3(1-e^{-u})(2e^{-2u}+e^{-u}+1)+\|w\|^2\mathbb{E}\|Z\|\sigma(u)\left[2e^{-2u}+2\right]\right.\nonumber\\
&\left.+\|w\|\sigma^2(u)\mathbb{E}\|Z\|^2\left[2(1-e^{-u})+1\right]+2\sigma^3(u)\mathbb{E}\|Z\|^3\right\rbrace\xrightarrow{u\searrow 0}0.\label{2}
\end{align}

Furthermore, given $\epsilon>0$, consider $R>0$ such that $\mathbb{P}(\|Z\|>R)<\epsilon$. Fix $\delta>0$, such that for any $a,b\in\mathbb{R}$: $|a-b|<\delta\Rightarrow|\sin(a)-\sin(b)|<\epsilon$. Now, for any $u$ such that $\sigma(u)R<\delta$ and for every $w\in D$, we have:
$$\|Z\|\leq R\Longrightarrow\left|\|we^{-u}+\sigma(u)Z\|-e^{-u}\|w\|\right|\leq \sigma(u)\|Z\|<\delta$$
and so:
\begin{align*}
&\left|\mathbb{E}\sin(\|we^{-u}+\sigma(u)Z\|)-\sin(e^{-u}\|w\|)\right|\\
\leq&\mathbb{E}\left|\sin(\|we^{-u}+\sigma(u)Z\|)-\sin(e^{-u}\|w\|)\right|\mathbb{1}\left[\|Z\|\leq R\right]\\
&+\mathbb{E}\left|\sin(\|we^{-u}+\sigma(u)Z\|)-\sin(e^{-u}\|w\|)\right|\mathbb{1}\left[\|Z\|> R\right]\\
\leq &\epsilon + 2\epsilon.
\end{align*}
Therefore:
\begin{align}
\sup_{w\in D}\left|\mathbb{E}\sin(\|we^{-u}+\sigma(u)Z\|)-\sin(e^{-u}\|w\|)\right|\xrightarrow{u\searrow 0} 0\label{3}.
\end{align}

Finally, for any $k\in \mathbb{N}$, consider $w_k\in D$ defined by $w_k(t)=k\pi$. For $u_k=-\log\left(1-\frac{1}{2k}\right)\xrightarrow{k\to\infty} 0$, we have:
$$\left|\sin(e^{-u_k}\|w\|)-\sin(\|w\|)\right|=\left|\sin\left(k\pi-\frac{\pi}{2}\right)-\sin(k\pi)\right|=1.$$
Therefore:
\begin{align}
\exists(u_k)_{k=1}^\infty:\quad u_k\xrightarrow{k\to\infty}0\quad\text{and}\quad \sup_{w\in D}\left|\sin(e^{-u_k}\|w\|)-\sin(\|w\|)\right|\geq 1.\label{4}
\end{align}

By (\ref{1}), (\ref{2}), (\ref{3}), (\ref{4}), $\lim_{u\to 0} \|T_uf-f\|_L\neq 0$ and so $T_u$ is not strongly continuous on $(L,\|\cdot\|_L)$.
\end{proof}
\section{Solution to the Stein equation}\label{sol}
We first show that the function, which in Lemma \ref{lemma3} is shown to solve the Stein equation, exists and belongs to the domain of $\mathcal{A}$.
\begin{lemma}\label{lemma2}
For any $g\in M$, such that $\mathbb{E}[g(Z)]=0$, $f=\phi(g)=-\int_0^\infty T_ugdu$ exists and is in the domain of $\mathcal{A}$.
\end{lemma}
\begin{proof}
Note that:
\begin{equation}\label{condition}
|g(w)-g(x)|\leq C_g(1+\|w\|^2+\|x\|^2)\|w-x\|
\end{equation}
uniformly in $w,x\in D[0,1]$. This follows from the fact that:
\begin{align*}
&|g(w)-g(x)|\leq \|g\|_M\|w-x\|^3+\left|Dg(x)[w-x]+\frac{1}{2}D^2g(x)[w-x,w-x]\right| \\
\leq& \|g\|_M\|w-x\|^3+\|Dg(x)\|\|w-x\|+\frac{1}{2}\|D^2g(x)\|\|w-x\|^2 \\
\leq& \|g\|_M\|w-x\|\left(\|w-x\|^2+1+\|x\|^2+\frac{1}{2}\|w-x\|(1+\|x\|)\right)\\
\leq& \|g\|_M\|w-x\|\left(2\|w\|^2+2\|x\|^2+1+\|x\|^2+\frac{1}{2}(\|w\|+\|x\|+\|w\|\|x\|+\|x\|^2)\right)\\
\leq& C_g(1+\|w\|^2+\|x\|^2)\|w-x\|
\end{align*}
uniformly in $w,x$ because $\|w\|\leq 1+\|w\|^2$, $\|x\|\leq 1+\|x\|^2$ and $\|w\|\|x\|\leq \|w\|^2+\|x\|^2$.
Now, we note that, as a consequence of (\ref{condition}), we have:
\begin{align}
&\lim_{t\to\infty}\int_0^t\left|T_ug(w)\right|du=\lim_{t\to\infty}\int_0^t\left|\mathbb{E}g(we^{-u}+\sigma(u)Z)\right|du\nonumber\\
\leq&\lim_{t\to\infty}\left[\int_0^t\left|\mathbb{E}\left[g(we^{-u}+\sigma(u)Z)-g(\sigma(u)Z)\right]\right|du+\int_0^t\left|\mathbb{E}[g(\sigma(u)Z)-g(Z)]\right|du\right]\nonumber\\
\leq &C_g\lim_{t\to\infty}\left[\int_0^t\mathbb{E}\left[\left(1+\|e^{-u}w+\sigma(u)Z\|^2+\sigma^2(u)\|Z\|^2\right)e^{-u}\|w\|\right]du\right.\nonumber\\
&\left.+\int_0^t\mathbb{E}\left|(1+(\sigma^2(u)+1)\|Z\|^2)\right|\left\|(\sigma(u)-1)Z\right\|du\right] \nonumber\\
\leq& C_g\lim_{t\to\infty}\left[\int_0^t\left[ e^{-u}\|w\|+2e^{-3u}\|w\|^3+3\sigma^2(u)e^{-u}\|w\|\mathbb{E}\|Z\|^2\right]du\right.\nonumber\\
&\left.+\int_0^t(\sigma(u)-1)\mathbb{E}\left|(1+(\sigma^2(u)+1)\|Z\|^2)\right|\left\|Z\right\|du\right]\nonumber\\
\leq&C(1+\|w\|^3)\text{,}
\label{4.77}
\end{align}
for some constant $C$.  Since $L$ is complete, this guarantees the existence of $\phi(g)$.

As noted in (2.23) and (2.24) of \cite{diffusion}, dominated convergence may be used, because of (\ref{4.77}) to obtain that:
\begin{equation}\label{4.800}
D^k\phi(g)(w)=-\int_0^\infty e^{-ku}D^kg(we^{-u}+\sigma(u)Z)du,\quad k=1,2.
\end{equation}
and, as a consequence, that $\phi(g)\in M$. This is enough to conclude that $\phi(g)$ belongs to the domain of $\mathcal{A}$ by the observation directly above the formulation of $\mathcal{A}$ labelled as (2.9) in \cite{diffusion}.
\end{proof}
\begin{remark}\label{remark}
The argument of (2.23) and (2.24) in \cite{diffusion} also readily gives that for any $g\in M$ and $t>0$: $\int_0^t T_ugdu\in M$.
\end{remark}
We now prove that observation (2.19) of \cite{diffusion} is true for all $g\in M$:
\begin{lemma}\label{lemma3} For all $t>0$ and for all $g\in M$:
\begin{equation}\label{4.100}
 T_{t}g-g=\mathcal{A}\left(\int_0^t T_ugdu\right)\text{.}
\end{equation}
\end{lemma}
\begin{proof}
We will follow the steps of the proof of Proposition 1.5 on p. 9 of \cite{ethier}. Observe that for all $w\in D[0,1]$ and $h>0$:
\begin{align}
&\frac{1}{h} [T_h-I]\int_0^tT_ug(w)du=\frac{1}{h}\int_0^t[T_{u+h}g(w)-T_ug(w)]du\nonumber\\
=&\frac{1}{h}\int_t^{t+h}T_ug(w)du-\frac{1}{h}\int_0^hT_ug(w)du\nonumber\\
\stackrel{(\ref{semigroup})}=&\frac{1}{h}\int_t^{t+h}\mathbb{E}[g(w e^{-u}+\sigma(u)Z)]du-\frac{1}{h}\int_0^{h}\mathbb{E}[g(w e^{-u}+\sigma(u)Z)]du.
&\label{4.600}
\end{align}
Taking $h\to 0$ on the left-hand side gives $\mathcal{A}\left(\int_0^t T_ug(w)du\right)$, since $\int_0^t T_ug(w)du$ belongs to the domain of $\mathcal{A}$ by Lemma \ref{lemma2} and Remark \ref{remark}. In order to analyse the right-hand side note that:
\begin{align}
&\left|\frac{1}{h}\int_0^h\mathbb{E}[g(w e^{-u}+\sigma(u)Z)]-g(w)du\right|\nonumber\\
\stackrel{\text{MVT}}\leq &\frac{1}{h}\int_0^h\mathbb{E}\left[\|w(e^{-u}-1)+\sigma(u)Z\|\sup_{c\in [0,1]}\|Dg\left(cw+(1-c)(we^{-u}+\sigma(u)Z)\right)\|\right]du\nonumber\\
\leq&\frac{\|g\|_M}{h}\int_0^h\mathbb{E}\left[\left(\|w\|(1-e^{-u})+\sigma(u)\|Z\|\right)\left(1+3\|w\|^2+3\|w\|^2e^{-2u}+3\sigma^2(u)\|Z\|^2\right)\right]du\nonumber\\
=&\frac{\|g\|_M}{h}\mathbb{E}\left\lbrace\left(1+3\|w\|^2+3\|Z\|^2\right)\left(\|w\|(-1+h+\cosh(h)-\sinh(h))\right.\right.\nonumber\\
&\left.+\|Z\|e^{-h}(-\sqrt{e^{2h}-1}+e^h(h+\log(1+e^{-h}\sqrt{-1+e^{2h}}))\right)\nonumber\\
&+3\|w\|(\|w\|^2-\|Z\|^2)\left(\frac{e^{-3h}}{6}(e^h-1)^2(e^h+2)\right)\nonumber\\
&+\left.3(\|w\|^2\|Z\|-\|Z\|^3)\frac{1}{3}\left(\sqrt{1-e^{-2h}}-\sqrt{e^{-6h}(e^{2h}-1)}\right)\right\rbrace\xrightarrow{h\to 0}0.
\label{4.200}
\end{align}

Similarly:
$$\left|\frac{1}{h}\int_t^{t+h}\mathbb{E}\left[g(we^{-u}+\sigma(u)Z)\right]du-\mathbb{E}\left[g(we^{-t}+\sigma(t)Z)\right]\right|\xrightarrow{h\to 0}0.$$
Therefore, as $h\to 0$, the right-hand side of (\ref{4.600}) converges to $T_tg-g$, which finishes the proof.
\end{proof}

\begin{proposition}\label{proposition}
For any $g\in M$, such that $\mathbb{E}g(Z)=0$, $f=\phi(g)=-\int_0^\infty T_ugdu$ solves the Stein equation:
$$\mathcal{A}f=g.$$
\end{proposition}
\begin{proof}
We note that for any $h>0$ and for any $f\in M$:
$$\frac{1}{h}\left[T_{n,s+h}f-T_sf\right]=T_s\left[\frac{T_h-I}{h}f\right].$$
We also note that for any $w\in D[0,1]$, $g\in M$ and some constant $K_1$ depending only on $f$:
\begin{align}
&\left|T_uf(w)-f(w)-\mathbb{E}Df(w)\left[\sigma(u)Z-w(1-e^{-u})\right]-\frac{1}{2}\mathbb{E}D^2f(w)\left[\left\lbrace\sigma(u)Z-w(1-e^{-u})\right\rbrace^{(2)}\right]\right|\nonumber\\
&\leq K_1(1+\|w\|^3)u^{3/2},\label{trala}
\end{align}
as noted on page 300 of \cite{diffusion}. Therefore, we can apply dominated convergence to obtain: 
\begin{align*}
\left(\frac{d}{ds}\right)^+T_sf(w)&=\lim_{h\searrow 0}T_s\left[\frac{T_h-I}{h}f(w)\right]=\lim_{h\searrow 0}\mathbb{E}\left[\frac{T_h-I}{h}f(we^{-s}+\sigma(s)Z)\right]\\
&=\mathbb{E}\left[\lim_{h\searrow 0}\frac{T_h-I}{h}f(we^{-s}+\sigma(s)Z)\right]=T_s\mathcal{A}f(w).
\end{align*}
Similarly, for $s>0$, $\left(\frac{d}{ds}\right)^-T_sf=T_s\mathcal{A}f$ because:
\begin{align*}
&\lim_{h\searrow0}\frac{1}{-h}\left[T_{s-h}f-T_sf\right](w)-T_s\mathcal{A}f(w)\\
=&\lim_{h\searrow0}T_{s-h}\left[\left(\frac{T_h-I}{h}-\mathcal{A}\right)f\right](w)+\lim_{h\searrow0}\left(T_{s-h}-T_s\right)\mathcal{A}f(w)\\
=&\lim_{h\searrow 0}\mathbb{E}\left[\left(\frac{T_h-I}{h}-\mathcal{A}\right)f(we^{-s+h}+\sigma(s-h)Z)\right]\\
&+\lim_{h\searrow 0}\mathbb{E}\left[\mathcal{A}f(we^{-s+h}+\sigma(s-h)Z)-\mathcal{A}f(we^{-s}+\sigma(s)Z)\right]\\
\stackrel{(\ref{trala})}=&0
\end{align*}
again, by dominated convergence. It can be applied because of (\ref{trala}) and the observation that for any $z\in D[0,1]$ and $h\in[0,1]$:
\begin{align*}
&\left|\mathcal{A}f(we^{-s+h}+\sigma(s-h)z)-\mathcal{A}f(we^{-s}+\sigma(s)z)\right|\\
=&\left|-Df(we^{-s+h}+\sigma(s-h)z)[we^{-s+h}+\sigma(s-h)z]\right.\\
&+\mathbb{E} D^2f(we^{-s+h}+\sigma(s-h)Z)[Z^{(2)}]\\
&\left.-Df(we^{-s}+\sigma(s)z)[we^{-s}+\sigma(s)z]+\mathbb{E} D^2f(we^{-s}+\sigma(s)z)[Z^{(2)}]\right|\\
\leq&\|f\|_M\left(1+\|we^{-s+h}+\sigma(s-h)z\|^2\right)\|we^{-s+h}+\sigma(s-h)z\|\\
&+\|f\|_M\left(1+\|we^{-s+h}+\sigma(s-h)z\|\right)\mathbb{E}\|Z\|^2\\
&+\|f\|_M\left(1+\|we^{-s}+\sigma(s)z\|\right)\|we^{-s}+\sigma(s)Z\|+\|f\|_M\left(1+\|we^{-s}+\sigma(s)z\|\right)\mathbb{E}\|Z\|^2\\
\leq&\|f\|_M\left(1+2\|w\|^2e^{-2s+2}+2\sigma^2(s-1)\|z\|^2\right)\left(\|w\|e^{-s+1}+\sigma(s-1)\|z\|\right)\\
&+\|f\|_M\left(1+\|we^{-s+1}+\sigma(s-1)z\|\right)\mathbb{E}\|Z\|^2\\
&+\|f\|_M\left(1+\|we^{-s}+\sigma(s)z\|\right)\|we^{-s}+\sigma(s)z\|+\|f\|_M\left(1+\|we^{-s}+\sigma(s)z\|\right)\mathbb{E}\|Z\|^2
\end{align*}
and so for any $h\in[0,1]$, $\left|\mathcal{A}f(we^{-s+h}+\sigma(s-h)Z)-\mathcal{A}f(we^{-s}+\sigma(s)Z)\right|$ is bounded by a random variable with finite expectation.

Thus, for all $w\in D[0,1]$ and $s>0$:
\begin{equation*}
\frac{d}{ds}T_sf(w)=T_s\mathcal{A}f(w)
\end{equation*}
and so, by the Fundamental Theorem of Calculus:
\begin{equation}\label{4.12}
T_rf(w)-f(w)=\int_0^rT_s\mathcal{A}f(w)ds.
\end{equation}
By Remark \ref{remark}, we can apply (\ref{4.12}) to $f=\int_0^tT_ugdu$ to obtain:
$$T_r\int_0^tT_ug(w)du-\int_0^tT_ug(w)du=\int_0^rT_s\mathcal{A}\left(\int_0^tT_ug(w)du\right)ds.$$
Now, we take $t\to\infty$. Let $Z'$ be an independent copy of $Z$. We apply dominated convergence, which is allowed because of (\ref{4.77}) and the following bound for $\varphi_t(w)=\int_0^tT_ug(w)du$:
\begin{align*}
&|\mathcal{A}\varphi_t(w)|\\
\leq&\int_0^t\mathbb{E}_Z\left|e^{-u}Dg(we^{-u}+\sigma(u)Z)[w]\right|du\\
&+\int_0^t\mathbb{E}_Z\left\lbrace\mathbb{E}_{Z'}\left|e^{-2u}D^2g(we^{-u}+\sigma(u)Z)\left[(Z')^{(2)}\right]\right|\right\rbrace du\\
\leq&\int_0^\infty\mathbb{E}_Z\left|e^{-u}Dg(we^{-u}+\sigma(u)Z)[w]\right|du\\
&+\int_0^\infty\mathbb{E}_Z\left\lbrace\mathbb{E}_{Z'}\left|e^{-2u}D^2g(we^{-u}+\sigma(u)Z)\left[(Z')^{(2)}\right]\right|\right\rbrace du\\
\leq& \|g\|_M\int_0^\infty e^{-u}\left(1+\mathbb{E}_Z\|we^{-u}+\sigma(u)Z\|^2\right)\|w\|du\\
&+\|g\|_M\int_0^\infty e^{-2u}\left(1+\mathbb{E}_Z\|we^{-u}+\sigma(u)Z\|\right)\mathbb{E}_{Z'}\|Z'\|^2du\\
\leq&\|g\|_M\int_0^\infty\left(e^{-u}+2\|w\|^2e^{-3u}+2\mathbb{E}_Z\|Z\|(e^{-u}-e^{-3u}\right)\|w\|du\\
&+\|g\|_M\int_0^\infty \left(e^{-2u}+\|w\|e^{-3u}+\sigma(u)e^{-2u}\right)\mathbb{E}_Z\|Z\|^2du\\
\leq&\left(1+\frac{4}{3}\mathbb{E}_Z\|Z\|^2\right)\|g\|_M(1+\|w\|^2)\|w\|+\left(\frac{1}{2}+\frac{\mathbb{E}_Z\|Z\|}{3}\right)\|g\|_M(1+\|w\|)\mathbb{E}_Z\|Z\|,
\end{align*}
where the second inequality follows again by dominated convergence applied because of (\ref{4.77}) in order to exchange integration and differentiation in a way similar to (\ref{4.800}). Then, we obtain:
\begin{align*}
T_r\int_0^\infty T_ug(w)du-\int_0^\infty T_ug(w)du&=\int_0^r T_s\lim_{t\to\infty}\mathcal{A}\left(\int_0^tT_ug(w)du\right)ds\\
&\stackrel{(\ref{4.100})}=-\int_0^r T_sg(w)ds.
\end{align*}
Now, by Lemma \ref{lemma2}, we can divide both sides by $r$ and take $r\to 0$ to obtain:
\begin{align*}
\mathcal{A}\left(\int_0^\infty T_ug(w)du\right)=-\lim_{r\to 0}\frac{1}{r}\int_0^r T_sg(w)ds
&=-\lim_{r\to 0}\left[\frac{1}{r}\int_0^r\mathbb{E}g(we^{-s}+\sigma(s)Z)ds\right]\\
&\stackrel{(\ref{4.200})}=-g(w),
\end{align*}
which finishes the proof.
\end{proof}
\begin{remark}
In \cite[Proposition 15]{vollmer} the authors prove that the semigroup of an $\mathbb{R}^d$-valued It\^o diffusion with Lipschitz drift and diffusion coefficients is strongly continuous on the space $L'=\left\lbrace x\mapsto (1+\|x\|^2)f(x): f\in C_0(\mathbb{R}^d)\right\rbrace $, equipped with the norm $\|f\|_{L'}=\sup_{x\in\mathbb{R}^d}\|f(x)\|_2/(1+\|x\|_2)$, where $C_0(\mathbb{R}^d)$ is the set of all continuous functions vanishing at infinity and $\|\cdot\|_2$ is the $l^2$ norm on $\mathbb{R}^d$. It might seem natural to try to adapt their argument to the infinite-dimensional setting and consider the space $L''=\left\lbrace w\mapsto (1+\|w\|^4)f(w): f\in C_0(D,\mathbb{R})\right\rbrace$, equipped with the norm $\|f\|_{L''}=\sup_{w\in D}|f(w)|/(1+\|w\|^4)$. Since $M\subset L''\subset L$, the semigroup \ref{semigroup} being strongly continuous on $L''$ would readily imply Proposition \ref{proposition}.

However, there is no easy extension of the argument used in the proof of \cite[Proposition 15]{vollmer} to the infinite dimensional setting. The reason is that the Riesz-Markov theorem for space $L'$ \cite[Theorem 2.4]{doersek} invoked in the proof, requires a closed unit ball in the domain of the functions in $L'$ to be compact. In other words, it requires the domain of the functions in $L'$ to be a finite-dimensional space. Since $D$ is infinite-dimensional, \cite[Theorem 2.4]{doersek} cannot be easily adapted to our setting and so the proof of \cite[Proposition 15]{vollmer} cannot be easily adapted either.
\end{remark}



\begin{thebibliography}{1}

\bibitem{barbour}
A.D. Barbour, \emph{{Stein's method and Poisson process convergence}}, Journal
  of Applied Probability \textbf{25} (1988), 175--184.

\bibitem{diffusion}
A.D. Barbour, \emph{{Stein's Method for Diffusion Approximations}}, Probability
  Theory and Related Fields \textbf{84} (1990), 297--322.

\bibitem{daprato1995ornstein}
Giuseppe Daprato and Alessandra Lunardi, \emph{{On the Ornstein-Uhlenbeck
  operator in spaces of continuous functions}}, Journal of Functional Analysis
  \textbf{131} (1995), no.~1, 94--114.

\bibitem{doersek}
P.~Doersek and J.~Teichmann, \emph{{A Semigroup Point of View On Splitting
  Schemes For Stochastic (Partial) Differential Equations}}, arXiv:1011.2651,
  2010.

\bibitem{ethier}
S.N. Ethier and T.G. Kurtz, \emph{Markov processes: characterization and
  convergence}, Wiley, New York, 1986.

\bibitem{vollmer}
J.~Gorham, A.B. Duncan, S.J. Vollmer, and L.~Mackey, \emph{{Measuring sample
  quality with diffusions}}, arXiv:1611.06972, 2016.

\bibitem{holmes}
S.~Holmes and G.~Reinert, \emph{Stein's method for bootstrap}, Lecture
  Notes-Monograph Series, vol.~46, Institute of Mathematical Statistics, 2004.

\bibitem{manca2008kolmogorov}
Luigi Manca, \emph{Kolmogorov operators in spaces of continuous functions and
  equations for measures}, Ph.D. thesis, Scuola Normale Superiore di Pisa,
  2008.

\end{thebibliography}


\providecommand{\bysame}{\leavevmode\hbox to3em{\hrulefill}\thinspace}
\providecommand{\MR}{\relax\ifhmode\unskip\space\fi MR }
\providecommand{\MRhref}[2]{%
  \href{http://www.ams.org/mathscinet-getitem?mr=#1}{#2}
}
\providecommand{\href}[2]{#2}


\ACKNO{We would like to thank Professor A.D. Barbour for an interesting discussion related to this work.}


\end{document}